\documentclass[12pt]{article} 

\usepackage[utf8]{inputenc} 

\usepackage{geometry} 
\usepackage{graphicx} 

\usepackage{booktabs} 
\usepackage{array} 
\usepackage{paralist} 
\usepackage{verbatim} 
\usepackage{subfig} 

\usepackage{fancyhdr} 
\usepackage{sectsty}
\allsectionsfont{\sffamily\mdseries\upshape} 

\usepackage[nottoc,notlof,notlot]{tocbibind} 
\usepackage[titles,subfigure]{tocloft} 

\usepackage{url} 
\usepackage{amsthm,amssymb,amsmath} 
\newtheorem{theorem}{Theorem}
\newtheorem{prop}{Proposition}

\title{On the Spectrum of Locally Linear Graphs}
\author{Reimbay Reimbayev}
\date{} 

\begin{document}
\maketitle

\begin{abstract}
For a locally linear graph $G$, which is a graph built out of triangles, it is possible to construct another graph $G^*$ that would consist of triangles of $G$ as vertices, while sharing (or not sharing) a common vertex between a pair of triangles would define a binary relation for edges of $G^*$. In this paper we show that the spectrum of $G^*$ is uniquely defined by $G$. We will also show some structural similarities of these graphs; in particular, that the number of quadrilaterals and pentagons in both graphs are the same; that $G^*$ does not contain $K_4-e$ and $K_{1,4}$; and that $G$ can be reconstructed from $G^*$.

\end{abstract}
\section*{Introduction}

An undirected graph $G$ is called locally linear if for any vertex $v \in V(G)$ an induced on its neighbors graph  $G[N(v)]$ is linear (regular of degree 1), where $N(v) = \{w\in V(G) | w\sim v, w \neq v)\}$ \cite{Dalibor}. Froncek \cite{Dalibor} showed that this definition is equivalent to the property that every edge of the graph $G$ with no isolated vertices belongs to a unique triangle. In that sense, connected locally linear graphs with no cycles, other than those triangles, are also called triangular cacti and studied in information science as networks immune to line failures \cite{Farley}. In extremal graph theory, locally linear graphs have gained a prominence with regard to one of the formulations of Ruzsa-Szemeredi problem, concerning the maximum number of edges in such graphs.

Thus locally linear graph $G$ is a graph that consists of triangles such that two triangles share at most one common vertex. This binary relation between any two triangles in $G$, which we  can call ``connected'' if they share one common vertex and ``not connected'' otherwise, naturally associates with $G$ a new graph, call it $G^*$. Vertex set of $G^*$ is the set of all triangles in $G$; and adjacency relationship is defined by connectedness property above, i.e. $x,y \in V(G^*): x\sim y$ if the corresponding triangles in $G$ share a common vertex.

To our best efforts, we were not able to identify the existing term, if there any, for such graphs. And although it is very tempting to call this graph $G^*$ with a special term, triangular-built, backbone, or something in that fashion, we will avoid adding a new nomenclature and will refer further on to likewise constructed graphs simply as $G^*$. In this paper we will study the properties of such graphs.

\section*{Forbidden Subgraphs in $G^*$}

First we consider two straightforward properties of $G^*$ answering to the question what kind of induced subgraphs it cannot contain. In a standard graph theory nomenclature, $K_4-e$, a complete graph on four vertices with an edge deleted, is called diamond \cite{GraphClasses}; $K_{1,4}$, a complete bipartite graph, is also denoted $S_4$, a four-star.

\begin{prop}
$G^*$ does not contain induced $K_4-e$.
\label{prop1}
\end{prop}

\begin{proof}
In a locally linear graph $G$ represented by $G^*$, any three mutually connected triangles are connected through a single common vertex. Otherwise there would be an edge that belongs to more than one triangle. Using this argument twice for $K_4-e$ configuration of four triangles in $G$ we conclude that all four triangles would have a common vertex.
\end{proof}

\begin{prop}
$G^*$ does not contain induced $K_{1,4}$.
\label{prop2}
\end{prop}

\begin{proof}
A triangle in $G$ can be connected to at most three other mutually disconnected triangles, as it has only three vertices. The fourth triangle would necessarily share the vertex with one of the previous three triangles , thus becoming connected to it.
\end{proof}

\noindent As a consequence of the two previous statements we have the next one.

\begin{prop}
Any two nonadjacent vertices of $G^*$ can have at most three common neighbors.
\label{prop3}
\end{prop}

\begin{proof}
Assume the opposite, i.e. there exist two nonadjacent vertices from $G^*$ with at least four common neighbors. Proposition \ref{prop2} guaranties that among any four common neighbors there will be at least one pair of adjacent vertices. This pair of adjacent neighbors and original two nonadjacent vertices then comprise a diamond, $K_4-e$, which according to Proposition \ref{prop1} is prohibited.
\end{proof}

\noindent We can say more about the structure of $G^*$. The next two propositions are dealing with invariants in $G$ and $G^*$ such as the number of cycles of particular length. Below, by quadrilaterals and pentagons we mean induced subgraphs isomorphic to cycles $C_4$ and $C_5$ respectively.

\begin{prop}
The graphs $G$ and $G^*$ have the same number of quadrilaterals.
\label{prop4}
\end{prop}

\begin{proof}
We prove the statement by showing a bijection from the set of all quadrilaterals in $G^*$ to the set of all quadrilaterals of $G$.

\textbf{One-to-one:} Given a quadrilateral $q^*=(x_1,x_2,x_3,x_4)$, where $x_i\in G^*$ and are triangles in $G$, denote their unique intersections in $G$: as $x_{1,2},x_{2,3},x_{3,4},x_{4,1}$. This vertices with four sides from each triangle make up a quadrilateral $q$ in $G$. And this is the only quadrilateral that is possible to construct on the sides of the given four triangles.

\textbf{Onto:} Given a quadrilateral $q$ from $G$, each edge of $q$ is a side of a distinct triangle (or else $q$ is not a quadrilateral). In order to show that these four triangles comprise a quadrilateral in $G^*$ we have to show that the opposite triangles are not connected. They are indeed cannot be connected, otherwise three mutually connected triangles would have a common vertex and two vertices of $q$ would collapse into one destroying the quadrilateral.
\end{proof}

\noindent Using similar arguments we can push a bit further and consider pentagons. Every pentagon in $G^*$ uniquely defines a sequence of triangles in $G$ connected through distinct vertices. Due to local linearity of $G$, those vertices comprising a closed cycle of length five do not contain any other ``idle'' vertices of triangles and are not adjacent to each other except of those that are already belong to the same side of a triangle. On the other hand every pentagon $C_5$ in $G$ defines a sequence of five distinct triangles comprising a cycle in $G^*$. None of the triangles are connected to none other than the neighboring two triangles. Thus the number of pentagons is an invariant as well and the statement below follows.

\begin{prop}
The graphs $G$ and $G^*$ have the same number of pentagons.
\label{prop5}
\end{prop}
 
\noindent Note that for cycles of length six and further it does not hold anymore. The relationship brakes up in both ways (Figure \ref{fig1}).

\begin{figure}
	\includegraphics[width=0.6\textwidth]{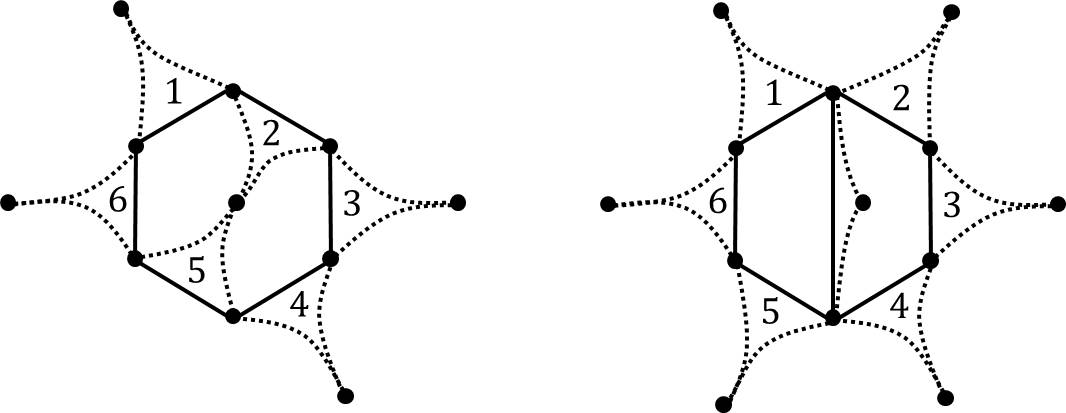}
		\centering
		\caption{Examples of hexagons that do not translate from $G$ to $G^*$ and vice versa. Left: triangles 1-2-3-4-5-6 do not comprise $C_6$ in $G^*$ while based on $C_6$ in $G$; Right: Triangles 1-2-3-4-5-6 although constituting $C_6$ in $G^*$ do not form $C_6$ in $G$.}
		\label{fig1}
\end{figure}

\section*{Characteristic Polynomial of $G^*$}

Now we will state the result regarding the spectrum of the graph $G^*$. Let us establish some notations first. We denote $A$ and $A^*$ - adjacency matrices of graphs $G$ and $G^*$ respectively; $D$ is a diagonal matrix with entries equal to the row (column) sums of $A$, equivalently - the degrees of corresponding vertices. $P_M(x)$ is a characteristic polynomial of a square matrix $M$.

\begin{theorem}
Let $|V(G)|=n$, and $|V(G^*)|=m$. Then:
\[P_{A^*}(x)=(x+3)^{m-n}P_{A+\frac{1}{2}D}(x+3).\]
In particular, when $G$ is a regular graph with valency $k$ (always even), so $nk = 6m$, the above formula becomes:
\[P_{A^*}(x)=(x+3)^{m-n}P_A (x-\frac{k}{2}+3).\]
\end{theorem}

\begin{proof}
For a locally linear graph $G$ enumerate, or label, all its vertices with numbers from 1 to $n$, and its triangles from 1 to $m$. Construct an $n \times m$ matrix $B$ in the following way:
\[ B_{i,j} =
  \begin{cases}
    1,  & \quad \text{if vertex } i \text{ contains in a triangle } j;\\
    0,  & \quad \text{otherwise.}
  \end{cases}
\]
This matrix $B$ is somewhat similar to an incidence matrix but instead of edges we are using triangles here. Now consider products $BB^T$ and $B^TB$.
\[(BB^T)_{i,j}=\sum_{k=1}^m B_{i,k} (B^T)_{k,j}=\sum_{k=1}^m B_{i,k} B_{j,k}=\]
\[=
\begin{cases}
	\text{number of triangles incident to a vertex $i$},   \text{if }i=j;\\
    1,   \quad \text{if }i\neq j\text{ and vertices } i  \text{ and } j \text{ both belong to the same triangle, i.e. }i \sim j ;\\
    0,   \quad \text{if }i\neq j\text{ and vertices } i  \text{ and } j \text{ do not belong to the same triangle, i.e. }i \nsim j .
 \end{cases}
 \]
Thus, $BB^T=A+\frac{1}{2}D$.

\noindent Similarly,
\[(B^TB)_{i,j}=\sum_{k=1}^n (B^T)_{i,k} B_{k,j}=\sum_{k=1}^n B_{k,i} B_{k,j}=\]
\[=
\begin{cases}
	3,   \quad \text{if }i=j: &\text{ triangles are the same};\\
    1,   \quad \text{if }i\neq j, i\sim j:&\text{ triangles are distinct and connected};\\
    0,   \quad \text{if }i\neq j, i\nsim j:&\text{ triangles are distinct and are not connected}.
 \end{cases}
 \]
 Thus, $B^TB=A^*+3I$.
 
 Using the fact from linear algebra that the products $BB^T$ and $B^TB$ for any matrix $B$ have the same non zero eigenvalues \cite{Horn, Cvetcovic}, we obtain:
 \[P_{B^TB}(x)=x^{m-n}P_{BB^T}(x).\]
 While,  \[P_{B^TB}(x)=P_{A^*+3I}(x)=P_{A^*}(x-3).\]
 So it becomes, 
 \[P_{A^*}(x)=(x+3)^{m-n}P_{A+\frac{1}{2}D}(x+3).\]
 When $G$ is $k$-regular, $D=kI$ and the result follows.
\end{proof}

Notice that $A+\frac{1}{2}D$ is somehow similar to signless Laplacian $Q=A+D$.

\section*{Reconstructability of the graph $G$ from $G^*$}

In the last section we will show that given a graph $G^*$ with properties described previously, most importantly by Proposition \ref{prop1} and Proposition \ref{prop2}, it can always be reconstructed the locally linear graph $G$.

\begin{theorem}
Given a graph $G^*$ with forbidden subgraphs $K_4-e$ and $K_{1,4}$, it is always possible to construct, or reconstruct, the unique, up to isomorphism, locally linear graph $G$.
\end{theorem}

\begin{proof}
Given a graph $G^*$ with vertices $x_1,x_2,...x_m$, for each $x_i$ we partition its neighboring vertices into sets of mutually adjacent vertices. We can do that because of Proposition \ref{prop1}. There will be at most three such sets (Proposition \ref{prop2}), some of them or even all three might be empty sets. We will add $x_i$ into each and arbitrarily label them as $V_{1,i},V_{2,i},V_{3,i}$.

We have:
\[\bigcup_{k=1}^3V_{k,i}=N(x_i)\cup\{x_i\}, \text{ and } V_{k,i}\cap V_{l,i}=\{x_i\}, \text{ for } k\neq l.\]
 
These sets $V_{k,i}$ ($i=\overline{1,m}, k=\overline{1,3}$) we declare a set of vertices of a being constructed graph $G$. Adjacency relationship we define as follows:
\[V_{k,i}\sim V_{l,j} \Leftrightarrow i=j.\]
To the moment $G=mK_3$, the graph consisting of $m$ disconnected graphs. But not all the sets $V_{k,i}$ are distinct, moreover two sets with more than one common elements will be identical. Let us show that. Assume $x,y\in V\cup W$, where $x,y$ are vertices of $G^*$ and $V,W$ are vertices of $G$. If there no other elements in $V$ and $W$, then we are done and they are identical. Otherwise, $z \in V$, which means $z\sim x$ and $z\sim y$. Now if $W$ have only two elements $x$ and $y$ then it has to have, by construction, $z$ as well. Or else, $w\in W$, distinct from $x,y$ and $z$. But then again, $z$ has to belong to $W$ or else we have a diamond, $K_4-e$, in $G^*$. Thus, $V\equiv W$, the sets are identical and we can identify (glue) them.
Summarizing so far, two subsets $V_1,V_2\subset V(G^*)$ as vertices of $G$ are: not adjacent if $|V_1\cap V_2|=0$, adjacent if $|V_1\cap V_2|=1$, identical if $|V_1\cap V_2|\geq 2$.

As a last step, we need to show that this process of gluing didn't create new triangles in $G$. In other words, the edges defined originally are still the sides of the initial triangles and only those ones - no edge, as a result of gluing we have performed, had become a side of more than one triangle. Assume opposite. $V_1\sim V_2$ and the edge $V_1V_2$ belongs to two triangles in $G$.

\textit{Case 1:} The edge belongs to the existing triangles build on $x$ and $y$ from $G^*$. But then both $x,y\in V_1$ and $x,y\in V_2$, which is not possible by construction. For $x$ as a neighbor of $y$ can belong to only one of $y$'s subsets.

\textit{Case 2:} The edge belongs now to a new triangles obtained as a result of gluing. Let $V_3$ is a third vertex of this triangle. Then, there exist three distinct $x,y$ and $z$ from $G^*$ such that $x\in V_1\cap V_2, y\in V_2\cap V_3$ and $z\in V_3\cap V_1$. This means that $x\sim y$ and $x\sim z$ and for the triples defined by $x\in G^*$ they are in a different sets of neighbors which means $y\nsim z$. On the other hand $y,z\in V_3$, meaning that $y\sim z$. Contradiction.

The graph $G$ obtained by our construction is locally linear and strictly defined by $G^*$ up to labeling.
\end{proof}

\subsection*{Appendix: Some Examples}

\begin{figure}
	\includegraphics[width=0.6\textwidth]{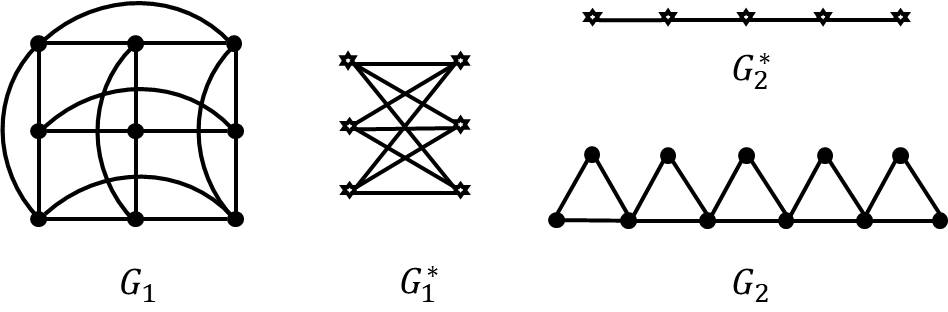}
		\centering
		\caption{Two examples of locally linear graphs and their $G^*$ representations.}
		\label{fig2}
\end{figure}

Consider two examples from Figure \ref{fig2}. $G_1$ is a Paley graph on 9 vertices , $n=9$, and 6 triangles, $m=6$. This is a strongly regular graph with parameters $k=4, \lambda=1,$ and $\mu=2$. Its representation $G_1^*=G^*(G_1)=K_{3,3}$ is a complete bipartite graph.
The characteristic polynomial of $G_1$:
\[P_{G_1}(x)=(x-4)(x-1)^4(x+2)^4\]
\noindent Then \[P_{G_1^*}(x)=(x+3)^{m-n}P_{G_1}(x-\frac{k}{2}+3)=\frac{(x-3)x^4(x+3)^4}{(x+3)^3}=x^4(x-3)(x+3).\]

Similarly, $G_2$, which sometimes is called a triangular snake, is a locally linear graph on 11 vertices and 5 triangles. Its representation $G_2^*$ now is a path on five vertices, $P_5$. Although it is not regular but still with recognizable spectrum. The characteristic polynomial of $A(G_2)$:
\[\begin{array}{ll} P_{G_2}(x)=&(x-3.027)(x-2.446)(x-1.631)(x-0.797)(x-0.201)(x+1)^2 \\ &(x+1.265)(x+1.37)(x+0.594)(x+1.872).
\end{array}\]
The spectrum for $A+\frac{1}{2}D$ will be much neater:
\[P_{A+\frac{1}{2}D}(x)=(x-4.732)(x-4)(x-3)(x-2)(x-1.268)x^6.\]

\noindent Then 
\[\begin{array}{ll} P_{G_2^*}(x)=&(x+3)^{m-n}P_{A+\frac{1}{2}D}(x+3)=\frac{(x-1.732)(x-1)x(x+1)(x+1.732)(x+3)^6}{(x+3)^6}= \\ &(x-1.732)(x-1)x(x+1)(x+1.732).
\end{array}\]


\end{document}